\documentclass[12pt,twoside,reqno]{amsart}
\usepackage{pst-plot}
\usepackage[latin1]{inputenc}
\usepackage[english]{babel}
\usepackage{amsmath, amsthm, amsfonts,amssymb}
\usepackage{color}
\usepackage{hyperref}

\usepackage{mathtools}
\mathtoolsset{showonlyrefs}

\usepackage[left=2.5cm,right=2.5cm,top=2cm,bottom=2cm]{geometry}
\newcommand\restr[2]{{
		\left.\kern-\nulldelimiterspace 
		#1 
		\vphantom{\big|} 
		\right|_{#2} 
}}
\usepackage{anysize}
\marginsize{10cm}{10cm}{2cm}{2cm}

\theoremstyle{plain}

\numberwithin{equation}{section}

\newtheorem{theorem}{Theorem}[section]
\newtheorem{lemma}[theorem]{Lemma}

\newtheorem{corollary}[theorem]{Corollary}

\newtheorem{proposition}[theorem]{Proposition}

\newtheorem{example}[theorem]{Example}

\def\ba{\begin{eqnarray*}}
	\def\ea{\end{eqnarray*}}

\def\bee{\begin{equation}}

\def\ene{\end{equation}}

\newcommand{\vertiii}[1]{{\left\vert\kern-0.25ex\left\vert\kern-0.25ex\left\vert #1 
		\right\vert\kern-0.25ex\right\vert\kern-0.25ex\right\vert}}

\usepackage{esint}

\begin{document}
\title[Nonlocal diffusion equations with dynamical boundary conditions]{Nonlocal diffusion equations with dynamical boundary conditions}
\author{Pablo M. Bern\'a}
\address{Pablo M. Bern\'a
	\\
	Departamento de Matem\'aticas
	\\
	Universidad Aut\'onoma de Madrid
	\\
	28049 Madrid, Spain} \email{pablo.berna@uam.es}

\author{Julio D. Rossi}
\address{Julio D. Rossi
	\\
	Departamento de Matem\'atica, FCEyN \\ Universidad de Buenos Aires\\ Buenos Aires, Argentina}
\email{jrossi@dm.uba.ar}



\begin{abstract}
In this paper we study nonlocal problems that are analogous to the
local ones given by the Laplacian or the $p-$Laplacian with dynamical boundary conditions.
We deal both with smooth and with singular kernels and show existence and uniqueness of solutions
and study their asymptotic behaviour as $t$ goes to infinity. 
\end{abstract}

\maketitle
\section{Introduction}
When one considers evolution problems for the usual Laplacian in a bounded smooth domain 
$\widehat{\Omega}$, among classical boundary conditions,
one finds the so-called dynamical boundary conditions. We refer to the following problem:  
\[
(DBC) \begin{cases}
\displaystyle 0=\Delta u(x,t), & x\in \widehat{\Omega}, \, t>0,\\[10pt]
\displaystyle  \frac{\partial u}{\partial t}(x,t) + \frac{\partial u}{\partial \eta} (x,t)=0, & x\in\partial \widehat{\Omega}, \, t>0,\\[10pt]
\displaystyle  u(x,0)=u_0(x),  & x\in\partial \widehat{\Omega}.
\end{cases} 
\]
Here, $\widehat{\Omega}$ is a bounded smooth domain in $\mathbb{R}^n$ and $\eta$ stands for the exterior unit normal on its boundary.
This is the natural evolution problem associated with the Dirichlet to Neumann map for the Laplacian in the domain
$\widehat{\Omega}$.
Notice that dynamical boundary conditions have been studied by many authors. Among them,
we mention: \cite{11,12,13,15,Latorre,Marcone,19,27}.

In this paper our main goal it to study nonlocal diffusion problems that are
analogous to the usual dynamical boundary conditions $(DBC)$.

\subsection{Linear equations with smooth kernels} First of all, we introduce what kind of nonlocal diffusion problems we will consider
here. To this end, we first take 
a non-singular kernel $J: \mathbb{R}^n \mapsto\mathbb R$, continuous, radially symmetric, decreasing, compactly supported (let $supp (J) = B_R(0)$) and non negative with $\int_{\mathbb{R}^n}J(w)dw = 1$. We also consider a fixed smooth domain $\widehat{\Omega}$ and inside this domain a narrow strip 
$\Gamma_r = \{ x\in \widehat{\Omega} : dist(x,\partial \widehat{\Omega})\leq r \}$, with $r\leq R$. We let
$\Omega_r = \widehat{\Omega} \setminus \Gamma_r$.
Associated with the kernel $J$ and the configuration $(\widehat{\Omega},\Gamma_r)$ we consider the nonlocal evolution problem:
\[
(P)\begin{cases}
\displaystyle 0=\int_{\Omega_r\cup \Gamma_r}J(x-y)\left[u(y,t)-u(x,t)\right]dy, & x\in\Omega_r, \, t>0,\\[10pt]
\displaystyle  \frac{\partial u}{\partial t}(x,t)=\int_{\Omega_r}J(x-y)\left[u(y,t)-u(x,t)\right]dy, & x\in\Gamma_r, \, t>0,\\[10pt]
\displaystyle  u(x,0)=u_0(x),  & x\in\Gamma_r.
\end{cases} 
\]
We will call this problem a nonlocal problem with dynamical boundary conditions.

At this point, let us mention that nonlocal equations of the form 
$$u_t(x,t)=\int J(x-y)\left[u(y,t)-u(x,t)\right]dy,$$
and variations of it, have been recently widely used in the modeling of diffusion processes. 
In a probabilistic interpretation of this model, $u(x,t)$ is thought as the density of a single population at the point $x$ and at time $t$, and $J(x-y)$ is the probability of jumping from location $y$ to location $x$. Then, $\int J(x-y) u(y,t) dy$ is the rate at which individuals are arriving to position $x$ from all other places and $- \int J(x-y) u(x,t) dy$ is the rate at which they are leaving location $x$ to travel to all other sites. Hence the equations says that the change in time of the density of individuals at $x$ at time $t$ is just the balance between arriving to/leaving from $x$ at time $t$ (see \cite{AMRT,CERW1,CERW2}).

Therefore, according to this probabilistic interpretation of the nonlocal terms, we can regard $(P)$ as a model
for the following situation: particles leave or arrive from $x\in \Omega_r$ to $y\in \Omega_r \cup \Gamma_r$
in very fast time scales (that are negligible compared with the slow time scales that occur in $\Gamma_r$)
given rise to an ``elliptic"{} equation inside $\Omega_r$ (notice that $t$ is only a parameter in the first equation that appears in $(P)$). 
On the other hand, individuals arrive to or leave from $x\in \Gamma_r$ from other sites $y\in \Omega_r$ in the slow time scale.
This gives the second equation in $(P)$ in which a time derivative appears.  

Also, notice that, given a function $u$ defined in $\Gamma_r$ in the first equation of $(P)$ we are solving 
an elliptic equation (that is like the local Laplacian) in $\Omega_r$ with exterior condition $u$ in $\Gamma_r$ and then, in the second equation in $(P)$, 
the right
hand side is just the flux from points inside $\Omega_r$ to the point $x\in \Gamma_r$. This is a nonlocal analogous to the Dirichlet to
Neumann map.

Our main result for $(P)$ is the following:

\begin{theorem} \label{teo.main.intro.J.2} {\it
		For every $1\leq p<\infty$, given $u_0\in L^p(\Gamma_r)$,  there exists a unique solution $u(x,t)
		\in \mathcal C([0,\infty); L^p(\Gamma_r))$ to $(P)$. Moreover, $u(x,t)$ preserves the total mass of the initial condition, that is,
		$$\int_{\Gamma_r}u(x,t)dx= \int_{\Gamma_r}u_0(x)dx, \qquad \mbox{for every } t>0.$$
		When the width of the strip $\Gamma_r$ is smaller that the support of $J$,  that is, when $r<R$,
		the asymptotic behaviour of the solutions as $t\to +\infty$ is given by the convergence to the mean value
		of the initial condition:
		for every $u_0\in L^p(\Gamma_r)$, there exists two constants $C(u_0)$ and $C'(\Gamma_r,p,\Omega_r)$
		such that
		$$\int_{{\Gamma_r}}\left\vert u(x,t)-\fint_{{\Gamma_r}}u_0\right\vert^p dx \leq C e^{-C't}.$$
	}
\end{theorem}

If we modify the second equation in problem $(P)$ considering as the second equation
$$\frac{\partial u}{\partial t} (x,t)=\int_{\Omega_r \cup \Gamma_r}J(x-y)(u(y,t)-u(x,t))dy,$$
we obtain
\[
(P^*)\begin{cases}
\displaystyle 0=\int_{\Omega_r\cup {\Gamma_r}}J(x-y)\left[u(y,t)-u(x,t)\right]dy, & x\in\Omega_r, \, t>0,\\[10pt] 
\displaystyle \frac{\partial u}{\partial t}(x,t)=\int_{\Omega_r\cup{\Gamma_r}}J(x-y)\left[u(y,t)-u(x,t)\right]dy, & x\in{\Gamma_r}
, \, t>0,\\[10pt] 
u(x,0)=u_0(x),  & x\in{\Gamma_r}.
\end{cases} 
\]

Here, as before, particles leave or arrive from $x\in \Omega_r$ to $y\in \Omega_r \cup \Gamma_r$
in very fast time scales, but individuals arrive to or leave from $x\in \Gamma_r$ from other sites $y\in \Omega_r\cup \Gamma_r$ 
in the slow time scale. Hence, the difference with the previous model is that now particles may jump from $\Gamma_r$ to $\Gamma_r$.

For this problem, we also have the same properties than for the problem $(P)$, the only difference being that
the asymptotic behaviour holds regardless the size of $\Gamma_r$. 

\begin{theorem} \label{teo.main.intro.J.2.99} {\it
		For every $1\leq p<\infty$, given $u_0\in L^p(\Gamma_r)$,  there exists a unique solution $u(x,t)
		\in  \mathcal C([0,\infty); L^p(\Gamma_r))$ of $(P^*)$. Moreover, $u(x,t)$ preserves the total mass of the initial condition and there exist two constants  $C=C(u_0)$ and $C'=C'(\Gamma_r,p,\Omega_r)$ such that the asymptotic behaviour of the solutions as $t\to +\infty$ is given by 
		$$\int_{{\Gamma_r}}\left\vert u(x,t)-\fint_{{\Gamma_r}}u_0\right\vert^p dx \leq C e^{-C't},$$
		for every $u_0\in L^p(\Gamma_r)$.
	}
\end{theorem}

\subsection{Nonlinear equations with smooth kernels} We can also extend our results to cover non linear nonlocal problems of $p-$Laplacian type.
Given $p>1$, we consider the problem 
\begin{equation} \label{p-lapla-intro}
(P2)\begin{cases}
\displaystyle 0=\int_{\Omega_r\cup {\Gamma_r}}J(x-y)\vert u(y,t)-u(x,t)\vert^{p-2}(u(y,t)-u(x,t))dy, & x\in\Omega_r, \, t>0,\\[10pt] 
\displaystyle \frac{\partial u}{\partial t}(x,t)=\int_{\Omega_r}J(x-y)\vert u(y,t)-u(x,t)\vert^{p-2} (u(y,t)-u(x,t)) dy, & x\in{\Gamma_r}, \, t>0,\\[10pt]
u(x,0)=u_0(x),  & x\in{\Gamma_r}.
\end{cases} 
\end{equation}

For references on nonlocal $p-$Laplacian problems we refer to \cite{AMRT,AMRT2,MRT,Vaz} and references therein.

For this problem we prove the following result:

\begin{theorem} \label{teo.main.intro.J.p} {\it
		Given $u_0\in L^p(\Gamma_r)$,  there exists a unique solution $u(x,t)
		\in \mathcal C((0,\infty); L^p(\Gamma_r))\cap W_{loc}^{1,1}((0,\infty); L^p(\Gamma_r))$of $(P2)$. Moreover, $u(x,t)$ preserves the total mass of the initial condition, that is,
		$$\int_{\Gamma_r}u(x,t)dx= \int_{\Gamma_r}u_0(x)dx, \qquad \mbox{for every } t>0.$$
		When the width of the strip $\Gamma_r$ is smaller that the support of $J$,  that is, when $r<R$,
		the asymptotic behaviour of the solutions as $t\to +\infty$ is given by the convergence to the mean value
		of the initial condition:  for every $1< p<\infty$, there exists a constant $C(J, \Gamma_r, p, u_0)$ such that 
			$$\int_{\Gamma_r}\left\vert u(x,t)-\fint_{{\Gamma_r}}u_0\right\vert^p dx \leq \dfrac{C}{t},$$
		and for every $1\leq q<\infty$ and $p>2$,
			$$\int_{\Gamma_r}\left\vert u(x,t)-\fint_{{\Gamma_r}}u_0 \right\vert^q dx \leq \dfrac{C}{t^{\alpha}},$$
			where $\alpha=\displaystyle \frac{p+q-2}{q}-1$ with $C=C(J,\Gamma_r,p,q,u_0)$.
	}
\end{theorem}

We can also consider the $p-$version of $(P^*)$, that is,
\begin{equation} \label{p-lapla-intro.*}
(P2^*)\begin{cases}
\displaystyle 0=\int_{\Omega_r\cup {\Gamma_r}}J(x-y)\vert u(y,t)-u(x,t)\vert^{p-2}(u(y,t)-u(x,t))dy, & x\in\Omega_r, \, t>0,\\[10pt] 
\displaystyle \frac{\partial u}{\partial t}(x,t)=\int_{\Omega_r \cup \Gamma_r}
J(x-y)\vert u(y,t)-u(x,t)\vert^{p-2} (u(y,t)-u(x,t)) dy, & x\in{\Gamma_r}, \, t>0,\\[10pt]
u(x,0)=u_0(x),  & x\in{\Gamma_r},
\end{cases} 
\end{equation}
and obtain similar results (whose proofs are left to the reader).

\subsection{Singular kernels}

We also include here the case in which the kernel $J$ can be singular. 
For $0<s<1$, we consider 
\begin{equation}\label{P3.intro.sing}
(P3)\begin{cases}
\displaystyle 0=\int_{\Omega_r\cup {\Gamma_r}}\frac{C(s)}{\vert x-y\vert^{n+ps}}\vert u(y,t)-u(x,t)\vert^{p-2}(u(y,t)-u(x,t))dy, & x\in\Omega_r, \, t>0,\\[10pt] 
\displaystyle \frac{\partial u}{\partial t}(x,t)=\int_{\Omega_r}\frac{C(s)}{\vert x-y\vert^{n+ps}}\vert u(y,t)-u(x,t)\vert^{p-2}(u(y,t)-u(x,t))dy, & 
x\in{\Gamma_r}, \, t>0, \\[10pt] 
u(x,0)=u_0(x),  & x\in{\Gamma_r},
\end{cases} 
\end{equation}
for some fixed constant $C(s)$ (the precise value of this constant is irrelevant here and is given in \cite{hich}). 

\begin{theorem} \label{teo.main.intro.sing.p}
{\it 	Given $u_0\in L^2(\Gamma_r)$,  there exists a unique solution $u(x,t) \in \mathbb \mathcal C([0,\infty); L^2(\Gamma_r))\cap W_{loc}^{1,1}((0,\infty); L^2(\Gamma_r))$ of $(P3)$. Moreover, $u(x,t)$ preserves the total mass of the initial condition, that is,
	$$\int_{\Gamma_r}u(x,t)dx= \int_{\Gamma_r}u_0(x)dx, \qquad \mbox{for every } t>0.$$
	The asymptotic behaviour of the solutions as $t\to +\infty$ is given by the convergence to the mean value
	of the initial condition:  if $u_0\in L^\infty(\Gamma_r)$ for $q>p$ or $u_0\in L^2(\Gamma_r)$ when $q=p$, then there exists $C=C(\Gamma_r, s, p,q,u_0)$ such that 
	$$\int_{\Gamma_r}\left\vert u(x,t)-\fint_{{\Gamma_r}}u_0\right\vert^q dx \leq \dfrac{C}{t}.$$}
\end{theorem}

Closely related to our results for the singular case are references \cite{delpezzo} and \cite{valdinoci}.
In \cite{valdinoci} the authors introduce Neumann boundary conditions
for fractional operators integrating in sets that are similar to ours. In \cite{delpezzo} the authors deal with nonlocal fractional
problems that approximate Steklov eigenvalues. For extra references concerning evolution problems
involving the fractional $p-$Laplacian we quote \cite{Ab,MRT,P,S,Vaz} and references therein.

\medskip

The paper is organized as follows:
in Section \ref{sec-2} we deal with the linear problem for a smooth kernel and prove Theorem
\ref{teo.main.intro.J.2}. The main argument to obtain the existence and uniqueness of solutions
is a fixed point argument. In Section \ref{sec-3}
we deal with the $p-$Laplacian with a smooth kernel.
Here our arguments rely on semigroup theory. In Section \ref{sec-4} we 
consider singular kernels (also using nonlinear semigroup theory). Finally, in Section \ref{app}
(the appendix), we collect the main notions and results concerning accretive operators that we use
to prove some of our main theorems.

\section{Existence, uniqueness and asymptotic behaviour for linear problems with smooth kernels. }
\label{sec-2}

\subsection{Existence and uniqueness.} Our main goal in this section is to prove the following existence and uniqueness 
result:
\begin{theorem}\label{main}
	{\it Given $u_0(x)\in L^p(\Gamma_r)$,  there exists a unique solution $u(x,t)\in \mathcal C((0,\infty); L^p(\Gamma_r))$ of $(P)$.}
\end{theorem}

Given $x\in\Omega_r$, the first equation in $(P)$ reads as
\begin{equation}\label{p2}
0=\int_{\Omega_r\cup\Gamma_r} J(x-y)u(y,t)dy - \left( \int_{\Omega_r\cup\Gamma_r}J(x-y)dy \right) u(x,t).
\end{equation}

To prove Theorem \ref{main}, we need the following technical lemma in which we show that given a function 
$g$ defined in $\Gamma_r$ we can solve the first equation in $(P)$, \eqref{p2}, inside $\Omega_r$
(with an estimate involving the corresponding $L^p$-norms). 

\begin{lemma}\label{lem}
	Given $g(y)\in L^p(\Gamma_r)$ for $t>0$, there exists a unique $u(x)=T(g)(y,t)\in L^p(\Omega_r)$ solution of \eqref{p2}
	with $u\equiv g$ in $\Gamma_r$. Moreover, for every $g \in L^p(\Gamma_r)$, we have 
	$$\Vert T(g)\Vert_{L^p(\Omega_r)}\leq \mathcal M\Vert g\Vert_{L^p(\Gamma_r)},$$
	where $\mathcal M=\mathcal M(\Omega_r, \Gamma_r, J, p)$.
\end{lemma} 
\begin{proof}
	Let $u\equiv g$ in $\Gamma_r$, and define the function $v=u-\hat{g}$, where 
	\[
	\hat{g}=\begin{cases}
	g & \text{in }\, \Gamma_r,\\[10pt]
	0& \text{in }\, \Omega_r.
	\end{cases} 
	\]
	Then, our problem in terms of $v$ reads as
	\begin{eqnarray*}
		\int_{\Omega_r\cup\Gamma_r}J(x-y)[v(y)-v(x)]dy &=& \underbrace{ \int_{\Omega_r\cup\Gamma_r}J(x-y)[u(y)-u(x)]dy}_{=0}\\ & & -\underbrace{\int_{\Omega_r\cup\Gamma_r}J(x-y)[\hat{g}(y)-\hat{g}(x)]dy}_{=f(x)}.
	\end{eqnarray*}
	Hence, we reformulate \eqref{p2} as
	\[
	(L)\begin{cases}
	\displaystyle 
	0= \int_{\Omega_r\cup \Gamma_r}J(x-y)\left[v(y)-v(x)\right]dy+f(x)& x\in\Omega_r, \\[10pt]
	v\equiv 0& x\in\Gamma_r.
	\end{cases} 
	\]

	Consider the variational problem associated with the previous equation,
	$$\min_{{v}|_{\Gamma_r} \equiv 0, v\in L^p(\Omega)}\frac{1}{4}\int_{\Omega_r\cup \Gamma_r}\int_{\Omega_r\cup \Gamma_r} J(x-y)\left[v(y)-v(x)\right]^2
	dydx - \int_{ \Omega_r}f(x)v(x)dx.$$
	We know that $$\int_{\Omega_r \cup \Gamma_r}\int_{\Omega_r \cup \Gamma_r} J(x-y)\left[v(y)-v(x)\right]^2
	dydx\approx \Vert v\Vert_{L^2(\Omega_r)}$$ for functions $v$ such that $v|_{\Gamma_r} =0$, see \cite{AMRT}. Hence, the variational problem has a unique minimum in $L^2(\Omega_r)$ since the associated functional is convex, coercive and bounded below.
	This minimum turns out to be the unique solution to $(L)$. Hence, we conclude that \eqref{p2} has a unique solution
	with $u\equiv g$ in $\Gamma_r$.
	
	Now our aim is to show the bound for $T(g)$.
	Recall that we called $v=T(g)-\tilde{g}$.
	Then, we have the following bound for the $L^p(\Omega_r)-$norm of the solution
	\begin{eqnarray*}
		\Vert T(g)\Vert_{L^p(\Omega_r)}^p &=& \Vert u(x)\Vert_{L^p(\Omega_r)}^p \lesssim \int_{\Omega_r}(f(x))^p dx\\
		&=&\int_{\Omega_r}\left\vert \int_{\Omega_r\cup\Gamma_r}J(x-y)(\hat{g}(y)-\hat{g}(x))dy\right\vert^pdx\\
		&\stackrel{\text{H\"{o}lder}}{\lesssim}& \int_{ \Omega_r}\int_{\Omega_r\cup \Gamma_r}J(x-y)\vert \hat{g}(y)\vert^pdydx\\
		&\leq&\int_{ \Omega_r}\int_{\Gamma_r}J(x-y)\vert \hat{g}(y)\vert^pdydx\\
		&\stackrel{\text{Fubbini}}{=}& K(J,p)\int_{\Gamma_r}\left(\int_{ \Omega_r}J(x-y)dx\right) \vert \hat{g}(y)\vert^pdy\\
		&\lesssim& \int_{\Gamma_r}\vert g(y)\vert^p dy.
	\end{eqnarray*}
	
	Thus, $\Vert T(g)\Vert_{L^p(\Omega_r)}\leq \mathcal M\Vert g\Vert_{L^p(\Gamma_r)}$.
\end{proof}

Now, we are ready to proceed with the proof of our existence and uniqueness result using a fixed point argument.

\begin{proof}[Proof of Theorem \ref{main}]
	If $x\in\Gamma_r$, it is easy to verify that integrating the second equation in $(P)$ we obtain
	$$u(x,t)-u_0(x) = \int_{0}^{t}\int_{\Omega_r}J(x-y)T(u)(y,s)dyds-\int_{0}^{t}a(x)u(x,s)ds,$$
	where $T$ is the operator that appears in our previous Lemma \ref{lem} and $a(x)= \int_{ \Omega}J(x-y)dy$. Hence, if we define
	$$H(u)(x,t)=u_0(x)+\int_{0}^{t}\int_{ \Omega_r} J(x-y)T(u)(y,s)dyds-\int_{0}^{t}a(x)u(x,s)ds,$$
	we have to show that $H$ has a unique fixed point in order to obtain a unique solution to $(P)$. 
	If we denote $$T'(u,v)(y,s)=T(u)(y,s)-T(v)(y,s)$$ and $$k(x,t)=u(x,t)-v(x,t),$$
	we get
	$$
	\begin{array}{l}
	\displaystyle	\vertiii{H(u)-H(v)}^p= \sup_{0\leq s\leq t}\Vert H(u)-H(v)\Vert_{L^p(\Gamma_r)}\\[10pt] 
	\displaystyle\quad= \sup_{0\leq s\leq t}\int_{\Gamma_r}\left\vert \int_{0}^t \int_{ \Omega_r} 
	J(x-y)T'(u,v)(y,s)dyds-\int_{0}^{t}a(x)k(x,s)ds\right\vert^pdx\\[10pt]
	\displaystyle\quad
	\leq  2^{p-1} \sup_{0\leq s\leq t}\int_{\Gamma_r}\left(\left\vert \int_{0}^t \int_{ \Omega_r} J(x-y)T'(u,v)(y,s)dyds\right\vert^p+\left\vert\int_{0}^{t}a(x)k(x,s)ds\right\vert^p\right)dx\\[10pt]
	\displaystyle\quad
	=(I)+(II),
	\end{array}
	$$
	where
	$$(I) = 2^{p-1}\sup_{0\leq s\leq t}\int_{\Gamma_r}\left\vert\int_0^t\int_{\Omega_r} J(x-y)T'(u,v)(y,s)dyds\right\vert^pdx,$$
	and
	$$(II)= 2^{p-1}\sup_{0\leq s\leq t}\int_{\Gamma_r}\left\vert\int_0^t a(x)k(x,s)dyds\right\vert^pdx.$$
	
	On the one hand,
	\begin{equation} \label{d1}
	\begin{array}{rl}
	(I)& \displaystyle\stackrel{\text{H\"{o}lder}}{\leq} t^{p/p'}K\sup_{0\leq s\leq t}\int_{\Gamma_r}\int_0^t\int_{ \Omega_r}J(x-y)\vert T'(u,v)(y,s)\vert^pdydsdx\\[10pt]
	& \qquad\displaystyle\leq t^{p/p'}K\sup_{0\leq s\leq t}\int_{0}^{t}\int_{ \Omega_r}\vert T'(u,v)(y,s)\vert^p dy ds\\[10pt] 
	\displaystyle
	&\qquad\displaystyle= t^{p/p'}K\sup_{0\leq s\leq t}\int_{0}^{t}\int_{ \Omega_r}\vert (T(u) - T(v))(y,s)\vert^p dy ds\\[10pt] 
	\displaystyle
	&\qquad\displaystyle \stackrel{\text{Lemma}\,\ref{lem}}{\leq}  t^{p/p'}K\sup_{0\leq s\leq t}\int_{0}^t\int_{\Gamma_r}\vert u-v\vert^pdyds\\[10pt]
	\displaystyle
	&\qquad\displaystyle = t^{p/p'+1}K \vertiii{u-v},
	\end{array}
	\end{equation}
	where $K=(J,\Omega_r,\Gamma_r,p)$. 
	
	On the other hand, 
	\begin{equation} \label{d2}
	\begin{array}{rl}
	(II)&
	\displaystyle\stackrel{\text{H\"{o}lder}}{\leq} Ct^{p/p'}\sup_{0\leq s\leq t}\int_{\Gamma_r}\int_{0}^ta(x)\vert k(x,s)\vert^pdsdx\\[10pt]
	&\qquad\displaystyle \leq Ct^{p/p'}\int_{\Gamma_r}\int_0^t \vert k(x,s)\vert^pdx 
	\\[10pt]
	& \qquad \displaystyle = C t^{p/p'+1}\vertiii{u-v},
	\end{array}
	\end{equation}
	where $C=C(J,\Omega_r,\Gamma_r,p)$. 
	
	Using \eqref{d1} and \eqref{d2}, we conclude that
	$$\vertiii{H(u)-H(v)}^p \leq K(J,\Omega_r,\Gamma_r,p)t^{p/p'+1}\vertiii{u-v}^p.$$
	
	Choosing now $t$ such that $K(J,\Omega_r,\Gamma_r,p)t^{p/p'+1}<1$, we conclude that $H$ is a strict contraction in 
	the Banach space $\mathbb X_t = \mathcal C([0,t]; L^p(\Gamma_r)) $ and the existence and uniqueness of the solution follows from Banach's fixed point theorem in the interval $[0,t]$. Notice that, from 
	$$u(x,t)-u_0(x) = \int_{0}^{t}\int_{\Omega_r}J(x-y)T(u)(y,s)dyds-\int_{0}^{t}a(x)u(x,s)ds,$$
	we deduce that $u(x,t)$ is differentiable with respect to $t$ a.e. and it is a solution (in a pointwise sense in $L^p$) to our original problem $(P)$.

	To extend the solution to $[0,\infty)$, we take as initial datum $u(x,t)\in L^p(\Gamma_r)$ and obtain a solution up to the interval $[t,2t]$ (and therefore we have a solution defined in $[0,2t]$). Hence, if we iterate this procedure we get a solution defined in the whole $[0,\infty)$.
\end{proof}

Now, we prove that $u(x,t)$ preserves the total mass of the initial condition.
\begin{proposition}\label{propmasa}
	For all $t>0$, it holds that
	$$\int_{\Gamma_r}u(x,t)dx= \int_{\Gamma_r}u_0(x)dx.$$
\end{proposition}

\begin{proof}
	Integrating the second equation in $(P)$ in $\Gamma_r$ we obtain
	\begin{eqnarray}\label{pr.1}
	\frac{\partial}{\partial t}\int_{\Gamma_r}u(x,t)dx &=& \int_{\Gamma_r}\frac{\partial u}{\partial t}(x,t)dx \\[10pt]
	&=& \int_{\Gamma_r}\int_{\Omega_r}J(x-y)\left[u(y,t)-u(x,t)\right]dydx\nonumber
	\end{eqnarray}
	On the other hand, integrating the first equation in $(P)$ in $\Omega_r$ we have
	\begin{equation}
	\begin{array}{rl}\label{pr.2}
	0 =& \displaystyle\int_{\Omega_r}\int_{\Omega_r\cup\Gamma_r}J(x-y)\left[u(y,t)-u(x,t)\right]dydx\\[10pt]
	=& \displaystyle \int_{\Omega_r}\int_{\Gamma_r}J(x-y)\left[u(y,t)-u(x,t)\right]dydx+\underbrace{\int_{\Omega_r}\int_{\Omega_r}J(x-y)\left[u(y,t)-u(x,t)\right]dydx}_{=0}\\[10pt]
	\stackrel{\text{Fubbini}}{=}&  \displaystyle \int_{\Gamma_r}\int_{ \Omega_r}J(x-y)\left[u(y,t)-u(x,t)\right]dydx.
	\end{array}
	\end{equation}
	
	Hence, from \eqref{pr.1} and \eqref{pr.2}, we conclude that
	$$\frac{\partial}{\partial t}\int_{\Gamma_r}u(x,t)dx=0,$$
	as we wanted to show.
\end{proof}

\subsection{Asymptotic behaviour}

In this section, we study the asymptotic behaviour as $t\longrightarrow\infty$ of the solution of the nonlocal problem $(P)$, that is, 
our goal is to prove that solutions converge to the mean value of the initial condition in $L^p-$norm, 
\begin{equation}\label{asym}
\int_{\Gamma_r} \left\vert u(x,t)-\fint_{{\Gamma_r}}u_0 \right\vert^pdx\longrightarrow 0, \qquad \mbox{ as } t \to +\infty.
\end{equation}

Taking $$v(x,t)=u(x,t)-\fint_{{\Gamma_r}}u_0,$$ 
we have that $v$ a solution to
\[
\begin{cases}
\displaystyle 0=\int_{\Omega_r\cup {\Gamma_r}}J(x-y)(v(y,t)-v(x,t))dy& x\in\Omega_r\\[10pt] 
\displaystyle \frac{\partial v}{\partial t}(x,t)=\int_{\Omega_r}J(x-y)\left[v(y,t)-v(x,t)\right]dy& x\in{\Gamma_r},
\end{cases} 
\]
with initial condition
$$v(x,0)= v_0 (x) =u_0 (x) -\fint_{{\Gamma_r}}u_0.$$ 
Notice that $$\fint_{\Gamma_r} v_0=0.$$

Taking into account that
$$\int_{{\Gamma_r}} v(x,t) v_t (x,t) \, dx = \int_{{\Gamma_r}}\int_{{\Gamma_r}} J(x-y)(v(y,t)-v(x,t))v(x,t)\, dydx,$$
we obtain
$$\frac{\partial}{\partial t}\left(\int_{{\Gamma_r}}\frac{v^2}{2} (x,t) \, dx\right)=-\frac{1}{2}\iint_H J(x-y)(v(y,t)-v(x,t))^2 \, dydx,$$
where $H=\left[({\Gamma_r}\cup\Omega_r)\times({\Gamma_r}\cup\Omega_r)\right]\setminus ({\Gamma_r}\times{\Gamma_r})$. 
In order to continue and show \eqref{asym}, we would like to have that
$$\beta:= \inf_{u\in L^2({\Gamma_r}),\int_{\Gamma_r} u=0}\frac{\displaystyle
	\frac{1}{2}\iint_H J(x-y)\vert u(y)-u(x)\vert^2dydx}{\displaystyle \int_{\Gamma_r} \vert u(x)\vert^2 dx}>0,$$
but this is false if we take ${\Gamma_r}$ with $r=R$ as the following example shows.

\begin{example}
	{\rm Consider a sequence $(f_n)_n$ such that 
	$$\int_{\Gamma_r} f_n^2 =1, \qquad \int_{\Gamma_r} f_n =0,$$ 
	and assume that the positive and negative parts of $f_n$ verify
	$$f_n = f_n^++f_n^-$$ with $$\int_{\Gamma_r} f_n^+=-\int_{\Gamma_r} f_n^-,$$ 
	$supp(f_n^+)\subset B(x_0; 1/n)$, $supp(f_n^-)\subset B(x_1;1/n)$ and $dist(x_0,\partial\Omega)=dist(x_1,\partial\Omega)=R$.
	
	Since $supp(f_n^+)\subset B(x_0; 1/n)$ and $dist(x_0,\partial\Omega)=R$ we have that there
	exists $\delta(1/n)$ such that $\delta (1/n)\to 0$ as $n\to \infty$ with
	$$
	\int_{\Omega_r} J(x-y) dy \leq \delta (1/n), \qquad \mbox{ for every } x\in \Gamma_r \cap B(x_0; 1/n).
	$$
	
	Then,
	$$
	\begin{array}{l}
	\displaystyle  \iint_HJ(x-y)\vert f_n(y)-f_n(x)\vert^2dydx \\[10pt]
	\displaystyle
	\qquad \leq \iint_H J(x-y)\vert f_n^+(y)\vert^2 dydx+ \iint_H J(x-y)\vert f_n^-(y)\vert^2 dydx
	\\[10pt]
	\displaystyle
	\qquad  = \int_{\Omega_r}\int_{B(x_0;1/n)} J(x-y) \vert f_n^+(y)\vert^2 dydx+\int_{\Omega_r}
	\int_{B(x_1;1/n)} J(x-y) \vert f_n^-(y)\vert^2 dydx
	\\[10pt]
	\displaystyle
	\qquad \leq \delta(1/n)\left[\int_{\Gamma_r}\vert f_n^-(y)\vert^2dy+\int_{\Gamma_r}\vert f_n^+(y)\vert^2dy\right].
	\end{array}
	$$
	Hence,
	$$\inf_{u\in L^2({\Gamma_r}),\int_{\Gamma_r} u=0} \frac{\displaystyle \iint_H J(x-y)\vert u(y)-u(x)\vert^2 dydx}{
		\displaystyle \int_{\Gamma_r} u^2}=0.$$}
\end{example}

Therefore, we need to impose that
$$
r<R
$$ 
to obtain the following result.

\begin{theorem}\label{main4}
	For any $1\leq p < \infty$, if $\Gamma_r$ is taken with $r<R$, 
	$$\beta:= \inf_{u\in L^p({{\Gamma_r}}),\int_{{\Gamma_r}} u=0}\dfrac{
		\displaystyle\iint_H J(x-y)\vert u(y)-u(x)\vert^p dydx}{ \displaystyle \int_{{{\Gamma_r}}}u^p(x)dx}>0,$$
	where $H=\left[(\Omega_r\cup{{\Gamma_r}})\times(\Omega_r \cup{{\Gamma_r}})\right]\setminus({{\Gamma_r}}\times{{\Gamma_r}})$.
\end{theorem}

\begin{proof}
	Arguing by contradiction, assume that $\beta=0$. Then, there is a sequence $u_n$ such that 
	$$\int_{{\Gamma_r}} u_n^p(x)dx=1, \qquad \int_{{\Gamma_r}} u_n(x)dx=0,$$ and 
	\begin{equation}\label{m1.99}
	\iint_H J(x-y)\vert u_n(y)-u_n(x)\vert^p dydx\longrightarrow 0.
	\end{equation}
	By \eqref{m1.99}, using \cite[Proposition 6.19]{AMRT}, there exists a constant $C=C(J, \Omega_r, p)$ such that
	$$0\longleftarrow \int_{\Omega_r}\int_{\Omega_r} J(x-y)\vert u_n(y)-u_n(x)\vert^p dydx \geq 
	C\int_{\Omega_r}\left\vert u_n(y)-\fint_{\Omega_r}u_n\right\vert^pdy,$$
	and moreover,
	$$\int_{{\Omega_r}} u_n^p(x)dx\leq C.$$
	Then, we can assume, extracting a subsequence if necessary, that $$\fint_{\Omega_r}u_n\longrightarrow A.$$ 
	Hence, we obtain that $u_n(y)\longrightarrow A$ strongly in $L^p(\Omega_r)$.

	Let us prove that
	\begin{equation} \label{uuu}
	\mathbb I:=\int_{{\Gamma_r}}\int_{\Omega_r} J(x-y)\vert A-u_n(x)\vert^pdydx \longrightarrow 0.
	\end{equation}
	Again, from \eqref{m1.99}, we have
	\begin{equation}\label{m2.99}
	\int_{{\Gamma_r}}\int_{\Omega_r} J(x-y)\vert u_n(y)-u_n(x)\vert^pdydx \longrightarrow 0.
	\end{equation}
	Now, taking into account the following bound,
	\begin{equation}\label{m4.99}
	\begin{array}{l}
	\displaystyle \mathbb I\leq 2^{p-1}\left[\int_{{\Gamma_r}}\int_{\Omega_r} J(x-y)\vert A-u_n(y)\vert^pdydx 
	\right. \\[10pt]
	\qquad \qquad \displaystyle \left.  +\int_{{\Gamma_r}}\int_{\Omega_r} J(x-y)\vert u_n(y)-u_n(x)\vert^pdydx\right],
	\end{array}
	\end{equation}
	using $u_n(y,t)\longrightarrow A$ strongly in $L^p(\Omega_r)$ and \eqref{m2.99} in \eqref{m4.99}, we obtain \eqref{uuu} .
	
	For $x\in{{\Gamma_r}}$, since $r<R$ we have 
	$$\int_{\Omega_r} J(x-y)dy\geq C>0.$$
	Then, using this uniform lower bound in \eqref{uuu}, we obtain
	\begin{eqnarray}\label{m5}
	\int_{{\Gamma_r}} \vert A-u_n(x)\vert^p dx\longrightarrow 0.
	\end{eqnarray} 
	
	Using that $u_n(x)\longrightarrow A$ in $L^p({{\Gamma_r}})$ and that $$\int_{{\Gamma_r}} u_n(x)dx =0,$$ we conclude that $A=0$.
	Hence, we have 
	$$\int_{{\Gamma_r}} \vert u_n(x)\vert^p dx = 1$$ and $$\int_{{\Gamma_r}} \vert u_n(x)-A\vert^p dx = \int_{{\Gamma_r}} \vert u_n(x)\vert^pdx\longrightarrow 0.$$ This is the desired contradiction.
\end{proof}

Now we are ready to proceed with the proof of the asymptotic behaviour of the solutions.

\begin{theorem}\label{mainp}
	Let $\beta$ be as in Theorem \ref{main4}. For every $1\leq p<\infty$ and $u_0\in L^p(\Gamma_r)$, it holds that
	$$\int_{{\Gamma_r}}\left\vert u(x,t)-\fint_{{\Gamma_r}}u_0(x)dx\right\vert^p\leq \left\| u_0-\fint_{{\Gamma_r}}u_0 \right\|_{L^p(\Gamma_r)}^pe^{-C'(p,\Gamma_r,\beta)t}.$$
\end{theorem}

\begin{proof}
First, we deal with the case $p=2$.
Thanks to 
$$\frac{\partial}{\partial t}\left(\int_{{\Gamma_r}}\frac{v^2}{2} (x,t) \, dx\right)=-\frac{1}{2}\iint_H J(x-y)(v(y,t)-v(x,t))^2 \, dydx,$$
and the previous result we obtain
	\begin{eqnarray}
	\frac{\partial}{\partial t}\left(\int_{\Gamma_r}\frac{\vert v(x)\vert^2}{2}dx\right) \leq -C(\Gamma_r,\beta)
	\left(\int_{\Gamma_r}\vert v(x)\vert^2dx \right),
	\end{eqnarray}
	and the proof follows just integrating this inequality.
	
	Now, for $p\neq 2$ we multiply the second equation in $(P)$ by
	$|v|^{p-2}v $ and, with similar computations, we obtain
	\begin{eqnarray}
	\frac{\partial}{\partial t}\left(\int_{\Gamma_r}\frac{\vert v(x)\vert^p}{p}dx\right) =
	- \frac{1}{2}\iint_H J(x-y)(v(y,t)-v(x,t)) ( |v|^{p-2}v (y,t) - |v|^{p-2}v(x,t)) \, dydx.
	\end{eqnarray}
	Now we use that
	for $1\leq p<\infty$ there exists a constant such that
	\begin{eqnarray}\label{ineq.55}
	(a-b)\left(\vert a\vert^{p-2}a -\vert b\vert^{p-2}b\right)\geq C\vert a -b\vert^p,
	\end{eqnarray}
	to obtain 
	\begin{eqnarray}
	\frac{\partial}{\partial t}\left(\int_{\Gamma_r}\frac{\vert v(x)\vert^p}{p}dx\right) =
	- \frac{1}{2}\iint_H J(x-y) ( v(y,t) - v(x,t))^p \, dydx.
	\end{eqnarray}
	Using again our previous result we get 
	\begin{eqnarray}
	\frac{\partial}{\partial t}\left(\int_{\Gamma_r}\frac{\vert v(x)\vert^p}{p}dx\right) \leq -C(\Gamma_r,p,\beta)
	\left(\int_{\Gamma_r}\vert v(x)\vert^p dx \right),
	\end{eqnarray}
	and the proof follows again just integrating this inequality.
\end{proof}

\subsection{A modification of problem $(P)$}

If we modify the second equation in problem $(P)$ considering
$$ \frac{\partial u}{\partial t} (x,t)=\int_{\Omega_r \cup {\Gamma_r}}J(x-y)(u(y,t)-u(x,t))dy,$$
we obtain
\[
(P^*)\begin{cases}
\displaystyle 0=\int_{\Omega_r\cup {{\Gamma_r}}}J(x-y)\left[u(y,t)-u(x,t)\right]dy, & x\in\Omega_r, \, t>0,\\[10pt] 
\displaystyle \frac{\partial u}{\partial t}(x,t)=\int_{\Omega_r\cup{{\Gamma_r}}}J(x-y)\left[u(y,t)-u(x,t)\right]dy, & x\in{{\Gamma_r}}
, \, t>0,\\[10pt] 
u(x,0)=u_0(x),  & x\in{{\Gamma_r}}.
\end{cases} 
\]

This new problem has a unique solution (this can be proved as we did for $(P)$ using a fixed point argument as we show below).

\begin{theorem}\label{main2}
	Given $u_0(x)\in L^p({{\Gamma_r}})$,  there exists a unique solution $u(x,t)\in \mathcal C([0,\infty); L^p(\Gamma_r))$ of ($P^*$).
\end{theorem}
\begin{proof}
	If $x\in{{\Gamma_r}}$, 
	\begin{eqnarray*}
		u(x,t)-u_0(x) &=& \int_{0}^{t}\int_{\Omega_r\cup{{\Gamma_r}}}J(x-y)\left[u(y,s)-u(x,s)\right]dyds\\[10pt]
		&=& \int_{0}^t\int_{\Omega_r\cup {{\Gamma_r}}}J(x-y)u(y,s)dyds-\int_{0}^{t}\underbrace{\left(\int_{\Omega_r\cup {{\Gamma_r}}}J(x-y)dy\right)}_{=a(x)}u(x,s)ds.
	\end{eqnarray*}
	
	Hence, if we define
	$$H(u)(x,t)=u_0(x)+\int_{0}^t\int_{\Omega_r\cup {{\Gamma_r}}}J(x-y)u(y,s)dyds-\int_{0}^{t}a(x)u(x,s)ds,$$
	we have to show that $H$ has a unique fixed point. As in Theorem \ref{main}, we will show that $H$ is a contraction. Taking $k(x,s)=u(x,s)-v(x,s)$,
	we obtain, using Lemma 3.10.3 of \cite{BL},
	$$
	\begin{array}{l}
	\displaystyle \vertiii{H(u)-H(v)}^p \displaystyle = \sup_{0\leq t\leq T}\Vert H(u)-H(v)\Vert_{L^p({{\Gamma_r}})}\\[10pt]
	\displaystyle  
	\qquad = \sup_{0\leq t\leq T}\int_{{{\Gamma_r}}}\left\vert \int_{0}^t \int_{ \Omega_r\cup{{\Gamma_r}}} J(x-y)\left[u(y,s)-v(y,s)\right]dyds-\int_{0}^{t}a(x)k(x,s)ds\right\vert^pdx \\[10pt]
	\displaystyle 
	\qquad \leq  2^{p-1} \sup_{0\leq t\leq T}\int_{{{\Gamma_r}}}\left\vert \int_{0}^t \int_{ \Omega_r\cup{{\Gamma_r}}} J(x-y)\left[u(y,s)-v(y,s)\right]dyds\right\vert^pdx
	\\[10pt]
	\displaystyle 
	\qquad \qquad +2^{p-1} \sup_{0\leq t\leq T}\int_{{{\Gamma_r}}}\left\vert\int_{0}^{t}a(x)k(x,s)ds\right\vert^p dx.
	\end{array}
	$$
	
	On the one hand,
	$$
	\begin{array}{l}
	\displaystyle	2^{p-1} \sup_{0\leq t\leq T}\int_{{{\Gamma_r}}}\left\vert\int_{0}^{t}a(x)k(x,s)ds\right\vert^p dx \\[10pt]
	\displaystyle \qquad	\leq CT^{p/p'}\sup_{0\leq t\leq T}\int_{{\Gamma_r}}\int_0^t\vert u(x,s)-v(x,s)\vert^pdsdx\\[10pt]
	\qquad \displaystyle \leq CT^{p/p'+1}\vertiii{u-v},
	\end{array}
	$$
	where $C=C(J, p, \Gamma, \Omega_r)$. 
	
	On the other hand, 
	\begin{eqnarray*}
		2^{p-1} \sup_{0\leq t\leq T}\int_{{{\Gamma_r}}}\left\vert \int_{0}^t \int_{ \Omega_r\cup{{\Gamma_r}}} J(x-y)\left[u(y,s)-v(y,s)\right]dyds\right\vert^pdx\leq (I+II),
	\end{eqnarray*}
	where $$I=2^{2p-2} \sup_{0\leq t\leq T}\int_{{\Gamma_r}}\left\vert \int_{0}^t \int_{{{\Gamma_r}}} J(x-y)\left[u(y,s)-v(y,s)\right]dyds\right\vert^pdx,$$
	and
	$$II=2^{2p-2} \sup_{0\leq t\leq T}\int_{{\Gamma_r}}\left\vert \int_{0}^t \int_{\Omega_r} J(x-y)\left[T(u)(y,s)-T(v)(y,s)\right]dyds\right\vert^pdx,$$
	with $T$ as in Lemma \ref{lem}.
	
	First we obtain the following bound for $I$,
	\begin{eqnarray*}
		I &\leq& 2^{2p-2} \sup_{0\leq t\leq T}CT^{p/p'}\int_{{\Gamma_r}}\int_{0}^{t}\int_{{\Gamma_r}} J(x-y)\vert u(y,s)-v(y,s)\vert^pdydsdx\\[10pt]
		&=&2^{2p-2}CT^{p/p'}\sup_{0\leq t\leq T}\int_{0}^{t}\int_{{\Gamma_r}}\int_{{\Gamma_r}} J(x-y)\vert u(y,s)-v(y,s)\vert^p dx dyds\\[10pt]
		&\leq& CT^{p/p'}\sup_{0\leq t\leq T}\int_{0}^{t}\int_{{\Gamma_r}} \vert u(y,s)-v(y,s)\vert^p\\[10pt]
		&\leq& CT^{p/p'+1}\vertiii{u-v},
	\end{eqnarray*}
where $C=C(J,\Omega_r,\Gamma_r,p)$. To obtain a bound for $ II$, using $T$ as in Lemma \ref{main}, we argue as follows,
	\begin{eqnarray*}
		II &\leq& 2^{2p-2}T^{p/p'} \sup_{0\leq t\leq T}\int_{{\Gamma_r}} \int_{0}^{t}\int_{\Omega_r} 
		J(x-y)\vert T(u)(y,s)-T(v)(y,s)\vert^p dy ds dx\\[10pt]
		&\leq& CT^{p/p'}\int_{0}^t \int_{\Omega_r}\vert T(u)(y,s)-T(v)(y,s)\vert^pdyds\\[10pt]
		&\stackrel{\text{Lemma}\,\ref{lem}}{\leq} & CT^{p/p'}\mathcal M\int_{0}^{t}\int_{{\Gamma_r}} \vert u(x,s)-v(x,s)\vert^pdxds\\[10pt]
		&\leq& T^{p/p'+1}C\vertiii{u-v},
	\end{eqnarray*}
	where $C=C(J,p,\Omega_r,\Gamma_r)$. 
	
	Then, there exists  $K=K(J,p,\Omega_r, \Gamma_r)$  such that $$\vertiii{H(u)-H(v)}\leq K T^{p/p'+1}\vertiii{u-v}.$$ 
	Choosing now $t$ such that $Kt^{p/p'+1}<1$, we conclude that $H$ is a strict contraction in $\mathbb X_t = \mathcal C([0,t]; L^p(\Gamma_r))$ and the existence and uniqueness of the solution follows from Banach's fixed point theorem in the interval $[0,t]$. To extend the solution to $[0,\infty)$, we take a solution data $u(x,t)\in L^p({{\Gamma_r}})$ and obtain a solution up to the interval $[0,2t]$. Hence, if we iterate this procedure we get a solution defined in $[0,\infty)$.	
\end{proof}

Now, we observe that Lemma \ref{lem} is also valid and we also have the conservation of the total mass in this case.

\begin{proposition}
	For every $t>0$,
	$$\int_{{{\Gamma_r}}}u(x,t) dx =\int_{{\Gamma_r}} u_0(x)dx.$$
\end{proposition}
\begin{proof}
	\begin{eqnarray*}
		\frac{\partial}{\partial t}\int_{{\Gamma_r}} u(x,t)dx &=& \int_{{\Gamma_r}}  \frac{\partial u}{\partial t}(x,t)dx = \int_{{\Gamma_r}}\int_{\Omega_r \cup {{\Gamma_r}}}J(x-y)\left[u(y,t)-u(x,t)\right]dydx\\
		&=& \underbrace{\int_{{\Gamma_r}}\int_{{\Gamma_r}} J(x-y)\left[u(y,t)-u(x,t)\right]dydx}_{=0}\\
		& & + \underbrace{\int_{{\Gamma_r}}\int_{\Omega_r} J(x-y)\left[u(y,t)-u(x,t)\right]dydx}_{=0\, \text{by Proposition \ref{propmasa}}}.
	\end{eqnarray*}
	Hence, $$\frac{\partial}{\partial t}\int_{{\Gamma_r}} u(x,t)dx=0.$$
\end{proof}

Finally, the asymptotic behaviour follows from our previous computations for $(P)$ since
we can obtain that 
$$\beta:= \inf_{u\in L^p({\Gamma_r}),\int_{\Gamma_r} u=0}\frac{\displaystyle
	\frac{1}{2}\int_{\Omega_r \cup {\Gamma_r}}\int_{\Omega_r \cup {\Gamma_r}} 
	J(x-y)\vert u(y)-u(x)\vert^p dydx}{\displaystyle \int_{\Gamma_r} \vert u(x)\vert^p dx}>0,$$
regardless the size of ${\Gamma_r}$ and the support of $J$. The details are left to the reader.

\section{Dynamical boundary conditions for the nonlocal $p-$Laplacian. }
\label{sec-3}

Given $p>1$ and a smooth kernel $J$, we consider the problem
\[
(P2)\begin{cases}
\displaystyle 0=\int_{\Omega_r\cup {{\Gamma_r}}}J(x-y)\vert u(y,t)-u(x,t)\vert^{p-2}(u(y,t)-u(x,t))dy, & x\in\Omega_r, \, t>0,\\[10pt] 
\displaystyle \frac{\partial u}{\partial t}(x,t)=\int_{\Omega_r}J(x-y)\vert u(y,t)-u(x,t)\vert^{p-2}(u(y,t)-u(x,t))dy, & x\in{{\Gamma_r}}, \, t>0,\\[10pt]
u(x,0)=u_0(x),  & x\in{{\Gamma_r}}.
\end{cases} 
\]

\subsection{Existence and uniqueness}

Our main goal is to prove the existence and uniqueness of solutions using abstract semigroup theory (see the Appendix).
To this end we need to write our problem as
$$
u_t + B^J_p (u) =0, \qquad u(0) =u_0,
$$
for a suitable operator $B^J_p$.

To go on with this tasc, we first need the following lemma, that provides the nonlinear version of Lemma \ref{lem}.

\begin{lemma}\label{lema}
	Given $g(y)\in L^{p}({\Gamma_r})$, there exists a unique solution $u(x)=T(g)(x)\in L^{p}(\Omega_r)$ such that
	\[
	\begin{cases}
	\displaystyle 0=\int_{\Omega_r\cup {\Gamma_r}}J(x-y)\vert u(y)-u(x)\vert^{p-2}(u(y)-u(x))dy, & x\in\Omega_r,\\[10pt]
	u\equiv g, & \text{in}\; \Gamma_r.
	\end{cases}
	\]
\end{lemma}

\begin{proof}
	The same arguments used in the proof of Lemma \ref{lem} works here.
\end{proof}

Using Lemma \ref{lema}, we write 
\[
\hat{u}(x)=\begin{cases}
T(u)(x), & x\in\Omega_r,\\[10pt] 
u(x),  & x\in{{\Gamma_r}}.
\end{cases} 
\]

We study the existence and uniqueness of solutions of the problem $(P2)$ from the point of view of Nonlinear Semigroup Theory. 
To this end, we introduce the operator $B_p^J$,
$$B_p^J(u)(x)=-\int_{\Omega_r} J(x-y)\vert \hat{u}(y)-u(x)\vert^{p-2}(\hat{u}(y)-u(x))dy,$$
that is well defined for $u\in L^p({\Gamma_r})$ and then we have that $L^p({\Gamma_r})\subset D(B_p^J)$.

\begin{lemma}\label{lem1}
	For every $u, v\in L^p({\Gamma_r})$,
	$$\int_{{\Gamma_r}}B^J_p(u)(x)v(x)dx = \frac{1}{2}\iint_H J(x-y)\vert \hat{u}(y)-u(x)\vert^{p-2}(\hat{u}(y)-u(x))(\hat v(y)-v(x))dydx,$$
	where, as before, $H=\left[(\Omega_r\cup{{\Gamma_r}})\times(\Omega_r\cup{{\Gamma_r}})\right]\setminus({{\Gamma_r}}\times{{\Gamma_r}})$.
\end{lemma}
\begin{proof}
	First, we know that
	\begin{eqnarray}\label{one}
	\int_{{\Gamma_r}}B^J_p(u)(x)v(x)dx = \underbrace{-\int_{{\Gamma_r}}\int_{\Omega_r} J(x-y)\vert \hat{u}(y)-u(x)\vert^{p-2}(\hat{u}(y)-u(x))v(x)dydx}_{\mathbb Y}.
	\end{eqnarray}
	Also, we have
	\begin{eqnarray*}
		0&=&\int_{\Omega_r}\int_{\Omega_r\cup {\Gamma_r}} J(x-y)\vert \hat{u}(y)-u(x)\vert^{p-2}(\hat{u}(y)-u(x))v(x)dydx\\[10pt]
		&=&\int_{\Omega_r\cup{\Gamma_r}}\int_{\Omega_r\cup{\Gamma_r}}J(x-y)\vert \hat{u}(y)-u(x)\vert^{p-2}(\hat{u}(y)-u(x))v(x)dydx\\[10pt]
		& & - \int_{{\Gamma_r}}\int_{\Omega_r}J(x-y)\vert \hat{u}(y)-u(x)\vert^{p-2}(\hat{u}(y)-u(x))v(x)dydx\\[10pt]
		& & - \int_{{\Gamma_r}}\int_{{\Gamma_r}}J(x-y)\vert \hat{u}(y)-u(x)\vert^{p-2}(\hat{u}(y)-u(x))v(x)dydx\\[10pt]
		&=& -\frac{1}{2}\int_H J(x-y)\vert \hat{u}(y)-u(x)\vert^{p-2}(\hat{u}(y)-u(x))(\hat v(y)-v(x))dydx\\[10pt]
		& & - \underbrace{\int_{{\Gamma_r}}\int_{\Omega_r} J(x-y)\vert \hat{u}(y)-u(x)\vert^{p-2}(\hat{u}(y)-u(x))v(x)dydx}_{-\mathbb Y}.
	\end{eqnarray*}
	The proof is complete.
\end{proof}

\begin{lemma}\label{lem2}
	Let $P:\mathbb R\rightarrow \mathbb R$ be a nondecreasing function. Then,
	\begin{itemize}
		\item[a)] for every $u,v\in L^p({\Gamma_r})$ such that $P(u-v)\in L^p({\Gamma_r})$, we have
		\begin{eqnarray}\label{eq1}
		\int_{{\Gamma_r}}(B_p^J(u)(x)-B_p^J(v)(x))P(u(x)-v(x))dx=A(x,y)
		\end{eqnarray}
		where
		\begin{eqnarray*}
			A(x,y)&=& \frac{1}{2}\iint_H J(x-y)(P(\hat{u}(y)-\hat{u}(y))-P(u(x)-v(x)))\\[10pt]
			& &\times(\vert \hat{u}(y)-u(x)\vert^{p-2}(\hat{u}(y)-u(x))-\vert \hat{v}(y)-v(x)\vert^{p-2}(\hat{v}(y)-v(x))).
		\end{eqnarray*}
		\item[b)] If $P$ is bounded, \eqref{eq1} holds for every $u, v\in D(B_p^J)$. 
	\end{itemize}
\end{lemma}
\begin{proof}
	To show a), we just observe that, since
	$$
	\begin{array}{l}
	\displaystyle \int_{{\Gamma_r}}(B_p^J(u)(x)-B_p^J(v)(x))P(u(x)-v(x))dx \\[10pt]
	\qquad \displaystyle = \int_{{\Gamma_r}}B_p^J(u)P(u-v)(x)dx - \int_{{\Gamma_r}}B_p^J(v)P(u-v)(x)dx,
	\end{array}
	$$
	applying Lemma \ref{lem1}, the result follows.
	
	If $P$ is bounded, b) follows from a). 
\end{proof}

\begin{theorem}\label{main6}
	The operator $B_p^J$ is completely accretive and satisfies the range condition $$L^p({\Gamma_r})\subset R(I+B_p^J).$$
\end{theorem}
\begin{proof}
	
	To prove this result, we will use some ideas of \cite[Theorem 6.7]{AMRT}. Given $u_i\in D(B_p^J)$, $i=1,2$ and $q\in P_0$, by the monotonicity Lemma \ref{lem2}, we have that
	$$\int_{{\Gamma_r}}(B_p^J(u_1)(x)-B_p^J(u_2)(x))q(u_1(x)-u_2(x))dx\geq 0,$$
	so $B_p^J$ is completely accretive operator.
	
	To show that $B_p^J$ verifies the range condition, we have to show that for every function $\phi\in L^p({\Gamma_r})$, there exists $u\in D(B_p^J)$ such that 
	\begin{eqnarray}
	u=(I+B_p^J)^{-1}\phi.
	\end{eqnarray}
	
	First, assume that $\phi\in L^\infty({\Gamma_r})$ and consider the following continuous monotone operator
	$$A_{n,m}(u):=T_c(u)+B_p^J(u)+\frac{1}{n}\vert u\vert^{p-2}u^+-\frac{1}{m}\vert u\vert^{p-2}u^-,$$
	where $T_c(a)=c\wedge(r\vee(-c))$, $c\geq 0$ and $a\in\mathbb R$. We have that $A_{n,m}$ is coercive in $L^p({\Gamma_r})$: 
	\begin{eqnarray*}
		\int_{{\Gamma_r}}A_{n,m}(u(x))u(x)dx&=&\underbrace{\int_{{\Gamma_r}}T_c(u(x))u(x)dx}_{\geq 0}+\int_{{\Gamma_r}}B_p(u)(x)u(x)dx\\[10pt]
		& & +\frac{1}{n}\int_{{{\Gamma_r}}} \vert u^+(x)\vert^p dx+\frac{1}{m}\int_{{{\Gamma_r}}}\vert u^-(x)\vert^pdx\\[10pt]
		&\geq& \underbrace{\frac{1}{2}\int_H J(x-y)\vert \hat{u}(y)-u(x)\vert^p dydx}_{\geq 0}+\min\left\lbrace \frac{1}{n},\frac{1}{m}\right\rbrace\int_{{\Gamma_r}}\vert u(x)\vert^pdx\\[10pt]
		&\geq& \min\left\lbrace \frac{1}{n},\frac{1}{m}\right\rbrace\int_{{\Gamma_r}}\vert u(x)\vert^pdx.
	\end{eqnarray*}
	
	Hence, 
	$$\lim_{\Vert u\Vert_{L^p({\Gamma_r})}\rightarrow+\infty}\dfrac{\displaystyle \int_{{\Gamma_r}}A_{n,m}(u)u}{\Vert u\Vert_{L^p({\Gamma_r})}}=+\infty,$$
	and we obtain that $A_{n,m}$ is coercive.
	
	Now, applying \cite[Corollary 30]{B}, there exists $u_{n,m}\in L^p({\Gamma_r})$ such that
	$$T_c(u_{n,m})+B_p^J(u_{n,m})+\frac{1}{n}\vert u_{n,m}\vert^{p-2}u_{n,m}^+-\frac{1}{m}\vert u_{n,m}\vert^{p-2}u_{n,m}^- = \phi.$$
	Using the monotonicity of $$B_p^J(u_{n,m})+\frac{1}{n}\vert u_{n,m}\vert^{p-2}u_{n,m}^+-\frac{1}{m}\vert u_{n,m}\vert^{p-2}u_{n,m}^-,$$ we obtain that $T_c(u_{n,m})\ll \phi$. Hence, if $c> \Vert \phi\Vert_{L^\infty({\Gamma_r})}$, from $u_{n,m}<<\phi$ we get $c> \Vert u_{n,m}\Vert_{L^\infty({\Gamma_r})}$, and we have
	$$u_{n,m}+B_p^J(u_{n,m})+\frac{1}{n}\vert u_{n,m}\vert^{p-2}u_{n,m}^+-\frac{1}{m}\vert u_{n,m}\vert^{p-2}u_{n,m}^- = \phi.$$
	Now, since $u_{n,m}$ is increasing in $n$ and decreasing in $m$, as $u_{n,m}\ll \phi$, we can pass to the limit as $n\rightarrow+\infty$ obtaining that $u_m$ is solution of
	$$u_m+B_p^J(u_{m})-\frac{1}{m}\vert u_{m}\vert^{p-2}u_{m}^- = \phi.$$
	Since $u_m$ is decreasing in $m$, passing to the limit as $m\rightarrow+\infty$, we obtain a solution to
	$$u+B_p^J(u)=\phi.$$
	This proves the range condition for $\phi \in L^\infty({\Gamma_r})$.
	
	To obtain the range condition for $\phi \in L^p({\Gamma_r})$, take $\phi\in L^p({\Gamma_r})$ and $\phi_n\in L^\infty({\Gamma_r})$ such that $\phi_n\rightarrow\phi$ in $L^p({\Gamma_r})$. Hence, by the previous case, $u_n =(I+B_p^J)^{-1}\phi_n$. Since $B_p^J$ is completely accretive, $u_n\rightarrow u$ in $L^p({\Gamma_r})$ and $B_p^J(u_n)\rightarrow B_p^J(u)$ in $L^{p'}({\Gamma_r})$. Then, we conclude that $u+B_p^J(u)=\phi$.
\end{proof}

Now we are ready to state and prove (as an immediate consequence of our previous results) our existence and uniqueness result.

\begin{theorem}\label{main5.77}{\it
		Suppose $p>1$ and let $u_0\in L^p({\Gamma_r})$. then, for any $T>0$ and $t<T$, there exists a unique solution $u(x,t)\in\mathcal C([0,\infty); L^p(\Gamma_r))\cap  W^{1,1}_{\text{loc}}((0,\infty);L^p(\Gamma_r)))$ to (P2).}
\end{theorem}

\begin{proof}[Proof of Theorem \ref{main5.77}] By \cite[Corollary A.52]{AMRT} (see Section \ref{app}), Theorem \ref{main5} is proved.
\end{proof}

Using the same ideas than in Proposition \ref{propmasa} (then we omit the proof), we can obtain easily that $u(x,t)$ preserves the total mass of the initial condition.
\begin{proposition}\label{propmasa.p}
	For all $t>0$,
	$$\int_{\Gamma_r}u(x,t)dx= \int_{\Gamma_r}u_0(x)dx.$$
\end{proposition}

\subsection{Asymptotic behaviour}

Our first result in this section shows that $u(x,t)$ converges to the mean value of the initial datum in $L^p$ 
(notice that $p$ is the exponent that appears in the equations in $(P2)$). We use again that the width of $\Gamma_r$ verifies $r<R$.

\begin{theorem}\label{mainn2} Fix $1<p<\infty$.
	If $\beta$ as in Theorem \ref{main4} and $u_0\in L^p({\Gamma_r}) \cap L^2 (\Gamma_r)$, then
	$$\int_{{\Gamma_r}}\left\vert u(x,t)-\fint_{{{\Gamma_r}}}u_0\right\vert^p dx \leq C\dfrac{\Vert u_0\Vert_{L^2({\Gamma_r})}^2}{t},$$
	where $C=C(J, p, \Omega_r, {\Gamma_r})$.
\end{theorem}
\begin{proof}
	Let $$w(x,t)=u(x,t)-\fint_{{{\Gamma_r}}}u_0.$$
	Since $u_0(x)\in L^p(\Gamma_r)$, we have that $u(x,t)\in L^p(\Gamma_r)$ and then $\Vert w(x,t)\Vert_{L^p(\Gamma_r)}<+\infty$. Therefore, 
	we obtain that $$\frac{\partial}{\partial t}\int_{{\Gamma_r}}\vert w(x,t)\vert^p dx = H(x,y),$$ where
	\begin{eqnarray*}
		H(x,y)&=&-\frac{p}{2}\iint_H J(x-y)\vert w(y,t)-w(x,t)\vert^{p-2}(w(y,t)-w(x,t))\\
		& & \qquad \qquad  \times(\vert w(y,t)\vert^{p-2}w(y,t)-\vert w(x,t)\vert^{p-2}w(x,t))dydx.
	\end{eqnarray*}
	Therefore, the $L^p({\Gamma_r})$-norm of $w(\cdot,t)$ is decreasing with $t$. Moreover, using Theorem \ref{main4},
	$$\int_{{\Gamma_r}}\vert w(x,t)\vert^pdx \leq \beta \iint_H J(x-y)\vert u(y,t)-u(x,t)\vert^p dydx.$$
	Consequently, 
	\begin{eqnarray*}
		t\int_{{\Gamma_r}}\vert w(x,t)\vert^p dx &\leq& \int_0^t\int_{{{\Gamma_r}}}\vert w(x,s)\vert^pdxdx\\
		&\leq& \beta\int_{0}^t \iint_H J(x-y)\vert u(y,s)-u(x,s)\vert^p dydxds.
	\end{eqnarray*}
	
	On the other hand, 
	$$\int_{{{\Gamma_r}}}\vert u(x,t)\vert^2dx-\int_{{{\Gamma_r}}} \vert u_0(x)\vert^2dx = -\int_0^t \iint_H J(x-y)\vert u(y,s)-u(x,s)\vert^pdydxds,$$
	which implies
	$$ \int_0^t\iint_H J(x-y))\vert u(y,s)-u(x,s)\vert^pdydxds\leq \Vert u_0\Vert_{L^2({\Gamma_r})}^2,$$
	and therefore we conclude that
	$$\int_{\Gamma_r} \vert w(x,t)\vert^pdx \leq \dfrac{C\Vert u_0\Vert_{L^2({\Gamma_r})}^2,}{t},$$
	where $C=C(J,p,\Omega_r,\Gamma_r)$.
\end{proof}

The previous theorem only deals with convergence in $L^p$-norm. 
In the following result, we give an extension  of this convergence taking into account other $L^q$-norms
and hence we have two different parameters, $q$ and $p$ (the exponent that appears in the equations).

\begin{theorem}\label{mainn}
	If $\beta$ as in Theorem \ref{main4} and $u(x,t)$ is the solution of the problem (P2), denoting by $w_0(x,t)=u_0(x)-\fint_{{{\Gamma_r}}} u_0$, for $1\leq q<\infty$ and $u_0\in L^q({\Gamma_r})$, 
	\begin{itemize}
		\item If $p>2$,
		$$\int_{{\Gamma_r}}\left\vert u(x,t)-\fint_{{{\Gamma_r}}}u_0\right\vert^q dx \leq \Vert w_0\Vert_{L^q({\Gamma_r})}^q\dfrac{C(p,q,{\Gamma_r},\beta)}{t^{\alpha}},$$
		where $\alpha=\displaystyle \frac{p+q-2}{q}-1>0$.
		\item If $p=2$,
		$$\int_{{{\Gamma_r}}}\left\vert u(x,t)-\fint_{{{\Gamma_r}}}u_0\right\vert^q dx
		\leq \Vert w_0\Vert_{L^q({\Gamma_r})}^qe^{-C'(p,{\Gamma_r},\beta)t}.$$
	\end{itemize}
\end{theorem}

\begin{proof}[Proof of Theorem \ref{mainn}]
For $p> 2$,	
	let $$w(x,t)=u(x,t)-\fint_{{{\Gamma_r}}}u_0.$$ Thus, $w(x,t)$ verifies $(P2)$ with $w(x,0)= w_0(x)$. We write $w(x):=w(x,t)$ and $w(y):=w(y,t)$. 
	
	On the one hand,
	\begin{eqnarray*}
		\frac{\partial}{\partial t}\left(\int_{{\Gamma_r}}\frac{\vert w(x)\vert^q}{q}dx\right) &=& \int_{{\Gamma_r}} \vert w(x)\vert^{q-2}w(x)w_t(x)dx\\[10pt]
		&=& \int_{{\Gamma_r}}\int_{\Omega_r} J(x-y)\vert w(y)-w(x)\vert^{p-2}(w(y)-w(x))\vert w(x)\vert^{q-2}w(x)dydx.
	\end{eqnarray*}
	
	On the other hand,
	$$
	\begin{array}{l}
	\displaystyle	0=\int_{\Omega_r\cup{{\Gamma_r}}}\int_{\Omega_r\cup{{\Gamma_r}}}J(x-y)\vert w(y,t)-w(x,t)\vert^{p-2}(w(y,t)-w(x,t))\vert w(x,t)\vert^{q-2}w(x,t)dydx
		\\[10pt]
		 \displaystyle \quad -\int_{{{\Gamma_r}}}\int_{{{\Gamma_r}}}J(x-y)\vert w(y,t)-w(x,t)\vert^{p-2}(w(y,t)-w(x,t))\vert w(x,t)\vert^{q-2}w(x,t)dydx\\[10pt]
		 \displaystyle\quad -\int_{{{\Gamma_r}}}\int_{ \Omega_r}J(x-y)\vert w(y,t)-w(x,t)\vert^{p-2}(w(y,t)-w(x,t))\vert w(x)\vert^{q-2}w(x,t)dydx\\[10pt]
		\displaystyle =-\frac{1}{2}\int_{\Omega_r \cup{{\Gamma_r}}}\int_{\Omega_r\cup{{\Gamma_r}}}J(x-y)\vert w(y,t)-w(x,t)\vert^{p-2}(w(y,t)-w(x,t) 
		\\[10pt]
		\displaystyle \qquad\qquad \qquad \qquad \times (\vert w\vert^{q-2}w(y,t)-\vert w\vert^{q-2}w(x,t))dydx\\[10pt]
		 \displaystyle \quad +\frac{1}{2}\int_{{{\Gamma_r}}}\int_{{{\Gamma_r}}}J(x-y)\vert w(y,t)-w(x,t)\vert^{p-2}(w(y,t)-w(x,t))
		 \\[10pt]
		\displaystyle \qquad \qquad \qquad \qquad \times(\vert w\vert^{q-2}w(y,t)-\vert w\vert^{q-2}w(x,t))dydx\\[10pt]
		 \displaystyle\quad  - \int_{{{\Gamma_r}}}\int_{ \Omega_r}J(x-y)\vert w(y,t)-w(x,t)\vert^{p-2}(w(y,t)-w(x,t))\vert w(x,t)\vert^{q-2}w(x,t)dydx\\[10pt]
		\displaystyle = - \frac{1}{2}\int_H J(x-y)\vert w(y,t)-w(x,t)\vert^{p-2}(w(y,t)-w(x,t))(\vert w\vert^{q-2}w(y,t)-\vert w\vert^{q-2}w(x,t))dydx\\[10pt]
		\displaystyle\quad  - \int_{{{\Gamma_r}}}\int_{ \Omega_r}J(x-y)\vert w(y,t)-w(x,t)\vert^{p-2}(w(y,t)-w(x,t))\vert w(x,t)\vert^{q-2}w(x,t)dydx.
	\end{array}
	$$
	
	Using now that, for $p\geq 2$ and $1\leq q<\infty$,
	\begin{eqnarray}\label{ineq}
	\vert a-b\vert^{p-2}(a-b)\left(\vert a\vert^{q-2}a -\vert b\vert^{q-2}b\right)\geq C\vert a^\gamma -b^\gamma\vert^p,
	\end{eqnarray}
	where $C=C(p,q)$ and $\gamma = \displaystyle \frac{p+q-2}{p}$,
 together with Theorem \ref{main4}, we obtain 
	\begin{eqnarray}\label{imp}
	\nonumber\frac{\partial}{\partial t}\left(\int_{{\Gamma_r}}\frac{\vert w(x,t)\vert^q}{q}dx\right) &\leq& -C(p,q)\beta\int_{{\Gamma_r}} \vert w(x,t)\vert^{p\gamma}dx 
	\\[10pt]
	\nonumber & = &-C(p,q)\beta\int_{{\Gamma_r}}\vert w(x,t)\vert^{q\frac{p+q-2}{q}}\\[10pt]
	&\stackrel{\text{H\"{o}lder}}{\leq}&-C(p,q,{\Gamma_r},\Omega_r, J)\left(\int_{{\Gamma_r}}\vert w(x,t)\vert^qdx\right)^{\frac{p+q-2}{q}}.
	\end{eqnarray}
	
	Integrating this differential inequality, we obtain that for $1\leq q<\infty$ and $p>2$,
	$$\int_{{\Gamma_r}}\left\vert w(x,t)\right\vert^qdx \leq  \Vert w_0\Vert_{L^q({\Gamma_r})}^q\dfrac{C(p,q,{\Gamma_r},\Omega_r, J)}{t^{\alpha}},$$
	where $\alpha=\displaystyle \frac{p+q-2}{q}-1$.
	
	The case $p=2$ was studied in Theorem \ref{mainp}.
\end{proof}

\section{Nonlocal nonlinear dynamical boundary conditions with a singular kernel}
\label{sec-4}

In this section we consider the same nonlocal problem but with a singular kernel with the same singularity that appears in the fractional $p-$Laplacian 
(see \cite{hich}) and study the following problem: for $0<s<1$,
\[
(P3)\begin{cases}
\displaystyle 0=\int_{\Omega_r\cup {{\Gamma_r}}}\frac{C(s)}{\vert x-y\vert^{n+ps}}\vert u(y,t)-u(x,t)\vert^{p-2}(u(y,t)-u(x,t))dy, & x\in\Omega_r, \, t>0,\\[10pt] 
\displaystyle \frac{\partial u}{\partial t}(x,t)=\int_{\Omega_r}\frac{C(s)}{\vert x-y\vert^{n+ps}}\vert u(y,t)-u(x,t)\vert^{p-2}(u(y,t)-u(x,t))dy & 
x\in{{\Gamma_r}}, \, t>0, \\[10pt] 
u(x,0)=u_0(x),  & x\in{{\Gamma_r}}.
\end{cases} 
\]

To deal with this problem, we consider the space
$$\mathbb X_{s,p}:=\Big\{ u\in L^p(\Omega_r\cup{\Gamma_r}) : \Vert u\Vert_{s,p}<+\infty\Big\},$$
where, 
$$\Vert u\Vert_{s,p}:= \left(\Vert u\Vert_{L^p({\Gamma_r})}^p+\iint_H\frac{C(s)}{\vert x-y\vert^{n+ps}}\vert u(y)-u(x)\vert^{p}dydx \right)^{1/p},$$
with, as before, $H=[(\Omega_r\cup{\Gamma_r})\times(\Omega_r\cup{\Gamma_r})\setminus ({\Gamma_r}\times{\Gamma_r})]$.

Again, we start solving the first equation in our problem.

\begin{lemma}\label{lemaa}
	Given $g$ in $\mathbb X_{s,p}$, there exists a unique $u\in \mathbb X_{s,p}$ solution of
	$$0=\int_{\Omega_r\cup{\Gamma_r}}\frac{C(s)}{\vert x-y\vert^{n+ps}}\vert u(y)-u(x)\vert^{p-2}(u(y)-u(x))dy,\qquad x\in\Omega_r,$$
	such that $u\equiv g$ in $\Gamma_r$.
\end{lemma}
\begin{proof}
	As in Lemma \ref{lema}, the existence of this element $u\in \mathbb X_{s,p}$ is given by the following minimization problem:
	$$\min_{u\equiv g\; \text{in}\; {\Gamma_r}, u\in\mathbb X_{s,p}}\frac{C(s)}{2p}\iint_H \frac{1}{\vert x-y\vert^{n+ps}}\vert u(y)-u(x)\vert^{p}dxdy.$$
\end{proof}

Notice that, thanks to this lemma, we can extend a function $g$ defined in $\Gamma_r$ to the whole $\Omega_r\cup \Gamma_r$ and, as before, 
we call this extension $\hat{u}$.

Now our aim is to apply Nonlinear Semigroup Theory (as we did in the previous section).
To this end, we consider the operator
$$B_p^s(u) (x)=-\int_{\Omega_r} \frac{C(s)}{\vert x-y\vert^{n+ps}}\vert \hat u(y)-u(x)\vert^{p-2}(\hat u(y)-u(x))dy,$$
where $\hat u(z)$ is defined as in Lemma \ref{lemaa}.

Now we can write our problem $(P3)$, as the following abstract Cauchy's problem:
\[
\begin{cases}
u_t (x,t) +B_p^s(u) (x,t)=0, & x\in{\Gamma_r},\, t>0, \\[10pt]
u(x,0)=u_0(x),& x\in{{\Gamma_r}}.
\end{cases} 
\]

The following two lemmas could be proved using the same techniques used in Lemmas \ref{lem1} and \ref{lem2}.

\begin{lemma}
	Given $u, v\in \mathbb X_{s,p}$,
	$$\int_{{{\Gamma_r}}} B_p^s(u(x))v(x)dx =\frac{1}{2}\iint_H \frac{C(s)}{\vert x-y\vert^{n+ps}}\vert \hat{u}(y)-u(x)\vert^{p-2}(\hat{u}(y)-u(x))(\hat v(y)-v(x))dydx.$$ 
\end{lemma}

\begin{lemma}\label{lem3}
	Let $P:\mathbb R\rightarrow \mathbb R$ be a nondecreasing functions. Then,
	\begin{itemize}
		\item[a)] for every $u,v\in \mathbb X_{s,p}$ such that $P(u-v)\in \mathbb X_{s,p}$, we have
		\begin{eqnarray}
		\int_{{\Gamma_r}}(B_p^s(u)(x)-B_p^s(v)(x))P(u(x)-v(x))dx=A(x,y)
		\end{eqnarray}
		where
		\begin{eqnarray*}
			A(x,y)&=& \frac{1}{2}\int_H \frac{C(s)}{\vert x-y\vert^{n+ps}}(P(\hat{u}(y)-\hat{u}(y))-P(u(x)-v(x)))\\
			& & \qquad \qquad \times (\vert \hat{u}(y)-u(x)\vert^{p-2}(\hat{u}(y)-u(x))-\vert \hat{v}(y)-v(x)\vert^{p-2}(\hat{v}(y)-v(x))).
		\end{eqnarray*}
		\item[b)] If $P$ is bounded, \eqref{eq1} holds for every $u, v\in D(B_p^s)$. 
	\end{itemize}
\end{lemma}

As before, with these two lemmas we can prove the following result,

\begin{theorem} \label{teo.river}
	In $L^2({\Gamma_r})$ the operator $B_p^s$ is completely accretive and satisfies the range condition $$L^2({\Gamma_r})\subset R(I+B_p^s).$$
	Moreover, $D(B_p^J)$ is dense in $L^2({\Gamma_r})$.
\end{theorem}
\begin{proof}
	As in Theorem \ref{main6}, using Lemma \ref{lem3}, we obtain that the operator is completely accretive. 
	
	To show the range condition, we proceed as follows.
	For $\phi\in L^2({\Gamma_r})$, consider the following variational problem
	$$\min_{u\in L^2({\Gamma_r})}\underbrace{\frac{1}{2p}\iint_H \frac{C(s)}{\vert x-y\vert^{n+p s}}\vert u(y)-u(x)\vert^pdxdy}_{\mathcal D_p^s(u)} +\frac{1}{2}\int_{{\Gamma_r}} u^2(x) dx -\int_{{\Gamma_r}}\phi(x)u(x)dx,$$
	and prove that there is a unique solution. Take a minimizing sequence $u_n\in \mathbb X_{s,p}$. We have that
	$$\mathcal D_p^s(u_n)+\frac{1}{2}\int_{{\Gamma_r}}\vert u_n(x)\vert^pdx -\int_{{\Gamma_r}}\phi(x)u_n(x)\leq M,\qquad \forall n\in\mathbb N.$$
	By Young's inequality, we have
	\begin{eqnarray}\label{young}
	\mathcal D_p^s(u_n)+\frac{1}{4}\int_{{\Gamma_r}}\vert u_n(x)\vert^pdx\leq M+4\int_{{\Gamma_r}} \vert \phi \vert^p,\; \forall n\in\mathbb N.
	\end{eqnarray}
	Therefore, $\Vert u_n\Vert_{s,q}\leq K$, $\forall n\in\mathbb N$. Thus, by the compact embedding theorem, we can assume that $u_n \to u$ in $L^p({\Gamma_r})$ and $u\in \mathbb X_{s,p}$. Moreover, we have that $u_n$ is bounded in $L^2({\Gamma_r})$ and consequently, $u\in L^2({\Gamma_r})$. By Fatou's lemma, we deduce that $u$ is actually a minimizer of the variational problem. The uniqueness follows by the strict convexity of the functional. 
	
	To show the range condition, we just observe that the minimizer $u$ verifies
	$$
	(I+B_p^s)^{-1} \phi = u.
	$$
	To see this fact, take a function $v\in \mathbb X_{s,p}$. Then, 
	$$\varphi(t):=\mathcal D_p^s(u+tv)+\frac{1}{2}\int_{{\Gamma_r}}(u+tv)^2 -\int_{{\Gamma_r}}\phi(u+tv)$$
	has a minimum at $t=0$, and, consequently, $\varphi' (0)=0$. 
	Computing this derivative we obtain that $
	(I+B_p^s)u= \phi $ and hence the range condition follows.
	
	Finally, we show that $D(B_p^J)$ is dense in $L^2({\Gamma_r})$. To this end, it is enough to show that 
	$$\mathbb X_{s,p} \subset \overline{D(B_p^J)}^{L^2({\Gamma_r})}.$$
	So, let us take $v\in \mathbb{X}_{s,p}$. By the range condition and the accretiveness of the operator, there exists $u_n$ such that
	$$
	\begin{array}{l}
\displaystyle	\frac{1}{2}\iint_{H}\frac{C(s)}{\vert x-y\vert^{n+ps}}\vert u_n(y)-u_n(x)\vert^{p-2}(u_n(y)-u_n(x))(\phi(y)-\phi(x))dydx \\[10pt]
\qquad \displaystyle = n\int_{{\Gamma_r}}(v(x)-u_n(x))\phi(x)dx,
\end{array}
$$
	for all $\phi\in \mathbb X_{s,p}$. Taking $\phi = v-u_n$ and applying the Young's inequality,
	$$
	\begin{array}{l}
\displaystyle		\int_{{\Gamma_r}}(v(x)-u_n(x))^2dx \\[10pt]
\displaystyle =\frac{1}{2n}\iint_H \frac{C(s)}{\vert x-y\vert^{n+ps}}\vert u_n(y)-u_n(x)\vert^{p-2}(u_n(y)-u_n(x))(v(y)-v(x))dydx\\[10pt]
\displaystyle
\quad - \frac{1}{2n}\iint_H \frac{C(s)}{\vert x-y\vert^{n+ps}}\vert u_n(y)-u_n(x)\vert^pdydx\\[10pt]
		\displaystyle \leq \frac{1}{2n}\frac{1}{p'}\iint_H \frac{C(s)}{\vert x-y\vert^{n+ps}}\vert u_n(y)-u_n(x)\vert^pdydx\\[10pt]
		\displaystyle \quad + \frac{1}{2n}\frac{1}{p}\iint_H \frac{C(s)}{\vert x-y\vert^{n+ps}}\vert v(y)-v(x)\vert^pdydx\\[10pt]
		\displaystyle \quad - \frac{1}{2n}\iint_H \frac{C(s)}{\vert x-y\vert^{n+ps}}\vert u_n(y)-u_n(x)\vert^pdydx\\[10pt]
		\displaystyle \leq\frac{1}{2n}\iint_H \frac{C(s)}{\vert x-y\vert^{n+ps}}\vert v(y)-v(x)\vert^pdydx,
	\end{array}
	$$
	from which it follows that $u_n\rightarrow v$ in $L^2({\Gamma_r})$.
\end{proof}

Again, as an immediate consequence of this result, we obtain existence and uniqueness of a solution to our problem.

\begin{theorem}\label{main5}{\it
		Suppose $p>1$ and let $u_0\in L^2({\Gamma_r})$. Then, for any $t>0$, there exists a unique solution $u(x,t)\in\mathcal C([0,\infty); L^2(\Gamma_r))\cap W^{1,1}_{\text{loc}}((0,\infty);L^2({\Gamma_r})))$ to $(P3)$.}
\end{theorem}

\begin{proof} Using now \cite[Corollary A.52]{AMRT}, for any $T>0$, there exists a solution of problem $(P3)$ with $u(x,t)\in \mathcal C([0,T); L^2({\Gamma_r}))\cap  W^{1,1}_{\text{loc}}((0,T);L^2({\Gamma_r})))$, 
for $t\in (0,T)$.
\end{proof}

We omit the proof of the following proposition since the same technique used in the proof of Proposition \ref{propmasa} also works here.
\begin{proposition}
	For all $t>0$,
	$$\int_{{\Gamma_r}} u(x,t)dx = \int_{{{\Gamma_r}}} u_0(x)dx.$$
\end{proposition}

\subsection{Asymptotic behaviour}
Now, we deal with the asymptotic behaviour of the solutions.

\begin{theorem}
	Let $q\geq p$ and $u(x,t)$ the solution of the Problem $(P3)$. For an initial datum $u_0\in L^\infty({\Gamma_r})$ if $q>p$ and $u_0\in L^2({\Gamma_r})$ if $q=p$, the $L^q$-norm of the solution goes to zero as $t\rightarrow \infty$ since we have the following estimate:
	$$\left\| u(\cdot,t)-\fint_{{{\Gamma_r}}}u_0\right\|_{L^q({\Gamma_r})}^q\leq C\dfrac{\Vert u_0-\fint_{{{\Gamma_r}}}u_0\Vert_{L^\infty({\Gamma_r})}^{q-p}\Vert u_0\Vert_{L^2({\Gamma_r})}^2}{t},\; \forall t>0,$$
	where $C=C({\Gamma_r},\Omega_r, n, s,p)$.
\end{theorem}

\begin{proof}
	Let $$w(x,t)=u(x,t)-\fint_{{{\Gamma_r}}}u_0.$$ 
	Since ${\Gamma_r}$ is bounded, in Theorem \ref{main4} we can replace $J(x-y)$ by the singular kernel and using that the solution preserves the total mass, we have that there exists a positive constant $C=C(n, s, \Omega_r, \Gamma_r)$ such that
	$$\int_{{\Gamma_r}}\vert w(x,t)\vert^pdx \leq C\iint_{H}\dfrac{\vert u(y,t)-u(x,t)\vert^p}{\vert x-y\vert^{n+sp}}dydx.$$ Hence,
	$$\int_{{{\Gamma_r}}}\vert w(x,t)\vert^qdx \leq C\Vert w_0\Vert_{L^{\infty}({\Gamma_r})}^{q-p}\iint_{H}\dfrac{\vert u(y,t)-u(x,t)\vert^p}{\vert x-y\vert^{n+sp}}dydx.$$
	
	Since the $L^q({\Gamma_r})$-norm of $u(\cdot,t)$ is decreasing  with $t$ (this fact can be proved as in Theorem \ref{mainn2}),
	\begin{eqnarray*}
		t\int_{{\Gamma_r}}\vert w(x,t)\vert^q &\leq& \int_0^t\int_{{\Gamma_r}}\vert w(x,s)\vert^q dxds\\
		&\leq& C\Vert w_0\Vert_{L^\infty({\Gamma_r})}^{q-p}\int_0^t\iint_{H}\dfrac{\vert u(y,t)-u(x,t)\vert^p}{\vert x-y\vert^{n+sp}}dydxds.
	\end{eqnarray*}
	
	As $$\frac{1}{2}\iint_{H}\dfrac{\vert u(y,t)-u(x,t)\vert^p}{\vert x-y\vert^{n+sp}}dydx=-\int_{{{\Gamma_r}}} \frac{\partial u}{\partial t}(x,t)u(x,t)dx,$$ integrating in space and time, we get
	$$\int_0^t\iint_{H} \dfrac{\vert u(y,t)-u(x,t)\vert^p}{\vert x-y\vert^{n+sp}}dydxds =\int_{{\Gamma_r}}\vert u_0(x)\vert^2dx-\int_{{\Gamma_r}}\vert u(x,t)\vert^2\leq \Vert u_0\Vert_{L^2({\Gamma_r})}^2.$$
	Thus,
	$$\int_{{\Gamma_r}}\vert w(x,t)\vert^q dx \leq C\dfrac{\Vert w_0\Vert_{L^\infty({\Gamma_r})}^{q-p}\Vert u_0\Vert_{L^2({\Gamma_r})}^2}{t},$$
	with $C=(\Gamma_r, \Omega_r, n, s, p)$.
\end{proof}

\section{Appendix}\label{app}
In this Appendix, we gather the main definitions and results about completely accretive operators that we have used in the previous sections.

Let $(\Gamma, \mathcal B, \mu)$ be a $\sigma$-finite measure space and let $M(\Gamma)$ denote the space of measurable functions from $\Gamma$ into $\mathbb R$. We denote by $L(\Gamma)$ the space $L^1(\Gamma)+L^\infty(\Gamma)$ and $$L_0(\Gamma):= \lbrace u\in L(\Gamma): \mu(\lbrace \vert u\vert>k\rbrace)<\infty\; \text{for any}\; k>0\rbrace.$$

For $u, v\in M(\Gamma)$, we write
$$u\ll v\; \text{if and only if}\; \int_\Gamma j(u)dx \leq \int_\Gamma j(v)dx,$$
for all $j\in J_0$, where
$$J_0=\lbrace j: \mathbb R\rightarrow [0,\infty],\; \text{convex},\, \text{l.s.c.},\, j(0)=0\rbrace,$$
where l.s.c in an abbreviation for lower semicontinuous.

Let $A$ be an operator in $M(\Gamma)$. We say that $A$ is completely accretive if
$$u-\hat u\ll u-\hat u+\lambda(u-\hat v),\; \forall \lambda>0\; \text{and all}\; (u,v),(\hat u, \hat v)\in A.$$
Let $$P_0 = \lbrace q\in \mathcal C^\infty(\mathbb F): 0\leq q'\leq 1,\; supp(q')\; \text{is compact and}\; 0\not\in supp(q)\rbrace.$$ The following result provides a useful characterization of the complete accretivity.

\begin{proposition}{\cite[Proposition A.42]{AMRT}}
	Let $u_0\in L_0(\Gamma)$, $v\in L(\Gamma)$. Then, 
	$$u\ll u+\lambda u,\; \forall \lambda>0 \Longleftrightarrow \int_\Gamma q(u)v\geq 0,\; \forall q\in P_0.$$
\end{proposition}

\begin{corollary}{\cite[Corollary A.43]{AMRT}}
	If $A\subseteq L^p(\Gamma)\times L^p(\Gamma)$, $1\leq p<\infty$, then A is completely accretive if and only if
	$$\int_\Gamma q(u-\hat u)(v-\hat v)\geq 0,\; \text{for any}\; q\in P_0, (u,v),(\hat u,\hat v)\in A.$$
\end{corollary}

\begin{corollary}{\cite[Corollary A.52]{AMRT}}
Suppose $\mu(\Gamma)<\infty$. If an operator $A\subseteq L^1(\Gamma)\times L^1(\Gamma)$ is an m-completely accretive operator in $L^1(\Gamma)$, then for every $u_0\in D(A)$, the mild solution $u(t)=e^{-tA}u_0$ of the problem
\[
\begin{cases}
\displaystyle 0\in \frac{du}{dt}+Au,\\[10pt]
u(0)=u_0,
\end{cases}
\]
is a strong solution, where $u$ is called a strong solution of an abstract Cauchy problem
\[
(CP)_{f,x}\begin{cases}
\displaystyle f(t)\in \frac{du}{dt} (t)+Au(t), & \text{on}\; t\in (0,T),\\[10pt]
u(0)=x,
\end{cases}
\]
where $f:(0,T)\rightarrow X$ and $A$ is an operator in $X$ with $X$ a Banach space, if 
\[
\begin{cases}
u\in\mathcal C([0,T]; X))\cap W_{loc}^{1,1}((0,T); X),\\[10pt]
\displaystyle f(t)\in \frac{du}{dt} +Au(t),& \text{a.e.}\; t\in (0,T),\\[10pt]
u(0)=x.
\end{cases}
\]
\end{corollary}

{\bf Acknowledgements} P. M. Bern\'a was partially supported by the grants MTM-2016-76566-P (MINECO, Spain) and 20906/PI/18 from Fundaci\'on S\'eneca (Regi\'on de Murcia, Spain). 

J. D. Rossi was partially supported by CONICET grant PIP GI No 11220150100036CO (Argentina) and by UBACyT grant 20020160100155BA 
(Argentina).

This project has received funding from the European Union's Horizon 2020 research and innovation programme under the Marie Sk\l odowska-Curie grant agreement No 777822.


\end{document}